\documentclass[12pt,leqno,amscd, amsfonts, amssymb,pstricks,verbatim]{amsart}

\usepackage{mathrsfs}
\usepackage{amsmath}
\usepackage{amssymb}
\usepackage{hyperref}

\newtheorem{thm}{Theorem}[section]
\newtheorem{lem}[thm]{Lemma}
\newtheorem{cor}[thm]{Corollary}
\newtheorem{prop}[thm]{Proposition}
\newtheorem{con}[thm]{Conjecture}

 \theoremstyle{definition}

\newtheorem{defn}[thm]{Definition}

\newtheorem{rem}[thm]{Remark}

\numberwithin{equation}{thm}

\def\N{\mathscr N}
\def\ggg{\mathfrak{g}}
\def\sss{\mathfrak{s}}
\def\lll{\mathfrak{l}}

\def\Aut {{\rm Aut\,}}
\def\max{{\rm max}}
\def\ad{{\rm ad}}
\def\id{{\textbf{id}}}

\begin{document}

\title[Borel subalgebras of the Witt algebra $W_1$]{Borel subalgebras of the Witt algebra $W_1$}

\author{Yu-Feng Yao and Hao Chang}

\address{Department of Mathematics, Shanghai Maritime University,
 Shanghai, 201306, China.}\email{yfyao@shmtu.edu.cn}

\address{Department of Mathematics, East China Normal University,
 Shanghai, 200241, China.} \email{hchang@ecnu.cn}



\thanks{This work is supported by the National Natural Science Foundation of China (Grant Nos. 11201293 and 11271130),
the Innovation Program of Shanghai Municipal Education Commission (Grant Nos. 13YZ077 and 12ZZ038), and the Fund of ECNU and SMU for Overseas Studies.}

\begin{abstract}
Let $\mathbb{F}$ be an algebraically closed field of characteristic $p>3$, and $\ggg$ the $p$-dimensional Witt algebra over $\mathbb{F}$.
Let $\N$ be the nilpotent cone of $\ggg$. Explicit description of $\N$ is given, so that the conjugacy classes of Borel subalgebras of
$\ggg$ under the automorphism group are determined. In contrast with only one conjugacy class of Borel subalgebras in a classical simple Lie
algebra, there are two conjugacy classes of Borel subalgebras in $\ggg$. The representatives of conjugacy classes of Borel subalgebras, i.e., the so-called
standard Borel subalgebras, are precisely given.
\end{abstract}

\maketitle

\section{Introduction}
As is well-known to all, Borel subalgebras play an important role in the structure and representation theory of a Lie algebra.
Let $\mathfrak{L}$ be a finite-dimensional simple Lie algebra over an algebraically closed
field of characteristic $0$. The famous structure theorem (cf. \cite{Hu}) asserts that there is only one conjugacy class of Borel subalgebras under the automorphism gourp of $\mathfrak{L}$. The same result is also true for a classical simple Lie algebra over an algebraically closed field of prime characteristic with
some mild restriction on the characteristic $p$. The classification theorem for finite-dimensional simple Lie algebras over an algebraically closed field of prime characteristic $p>5$ implies that each finite-dimensional simple Lie algebra is either of classical type or of Cartan type (cf. \cite{BW, PS}). In this paper, we initiate the study of Borel subalgebras in the simple Lie algebras of Cartan type. We completely determine the conjugacy classes of Borel subalgebras in the Witt algebra, which is the ``simplest" Lie algebra of Cartan type.

Let $\ggg=W_1$ be the Witt algebra which was found by
Witt as the first example of non-classical simple Lie algebra in 1930s. As is known to all, $\ggg$
is a restricted Lie algebra, and has a natural $\mathbb{Z}$-grading $\ggg=\sum_{i=-1}^{p-2}\ggg_{[i]}$. Associated with this grading, one has
a filtration $(\ggg_i)_{i\geq -1}$ with $\ggg_i=\sum_{j\geq i}\ggg_{[j]}$ for $i\geq -1$. Let ${\N}=\{x\in\ggg\mid x^{[p]}=0\}$ be the nilpotent
cone of $\ggg$, which is a closed subvariety in $\ggg$. A. Premet extensively studied $\N$ in \cite{Pr}, where he gave
more general results on the nilpotent cone and the Chevalley Restriction Theorem for the Jacobson-Witt algebra. Quite recently,
Premet's results were set up in the four classes of Lie algebras of Cartan type in \cite{BFS}. Precise decomposition of $\N$ into distinct nilpotent orbits under the automorphism group of $\ggg$ was given in \cite{YS}. In contrast with finitely many nilpotent orbits in a classical simple Lie algebra (cf. \cite{Jan-1}), there are infinitely many nilpotent orbits in the Witt algebra. Moreover, M. Mygind \cite{My} provided a complete picture of the orbit closures in the Witt algebra and its dual space, extending the results in \cite{YS}.

This paper is structured as follows. After recalling some basic definitions and results on the Witt algebra, we explicitly describe nilpotent
elements not contained in $\ggg_0$ in section 3. We give a sufficient and necessary condition for a nilpotent element not contained in $\ggg_0$. This
detailed description on nilpotent elements helps us to give a complete classification of conjugacy classes of Borel subalgebras of $\ggg$ in the final
section. Our main result asserts that there are two conjugacy classes of Borel subalgebras in total. The representatives and dimensions are precisely given.

\section{Preliminaries}
In this paper, we always assume that the ground field $\mathbb{F}$ is algebraically closed, and of characteristic $p>3$. Let
${\mathfrak{A}}={\mathbb{F}}[x]/(x^p)$ be the truncated polynomial algebra of one indeterminate, where $(x^p)$ denotes the ideal of ${\mathbb{F}}[x]$
generated by $x^p$. For brevity, we also denote by $x$ the coset of $x$ in $\mathfrak{A}$. There is a canonical basis $\{1,x,\cdots, x^{p-1}\}$ in
$\mathfrak{A}$. Let $D$ be the linear operator on $\mathfrak{A}$ subject to the rule $Dx^i=ix^{i-1}$ for $0\leq i\leq p-1$. Denote by $W_1$ the derivation
algebra of $\mathfrak{A}$. In the following, we always assume $\ggg=W_1$ unless otherwise stated. By \cite[\S\,4.2]{SF},
$\ggg=\text{span}_{\mathbb{F}}\{x^iD\mid 0\leq i\leq p-1\}$. There is a natural $\mathbb{Z}$-grading on $\ggg$, i.e., $\ggg=\sum_{i=-1}^{p-2}\ggg_{[i]}$,
where $\ggg_{[i]}={\mathbb{F}}x^{i+1}D,\,-1\leq i\leq p-2$. Associated with this grading, one has the following natural filtration:
$$\ggg=\ggg_{-1}\supset\ggg_0\supset\cdots\supset\ggg_{p-2}\supset 0,$$
where $$\ggg_i=\sum\limits_{j\geq i}\ggg_{[j]}, \,-1\leq i\leq p-2.$$
This filtration is preserved under the action of the automorphism group $G$ of $\ggg$ (cf. \cite{Ch, Re, Wi}). Furthermore, $\ggg$ is a restricted
Lie algebra with the $[p]$-mapping defined as the $p$-th power as usual derivations. Precisely speaking,
$$(x^iD)^{[p]}=
\begin{cases}
0, &\text{if}\,\,i\neq 1, \cr xD, &\text{if}\,\, i=1.
\end{cases}$$
We need the following result on the automorphism group of $\ggg$.

\begin{lem}(cf. \cite{Ch, Wi}, see also \cite[Theorem 12.8]{Re})\label{lem-1}
Let $\ggg=W_1$ be the Witt algebra over $\mathbb{F}$ and $G=\Aut(\ggg)$. Then the following statements hold.
\begin{itemize}
\item[(i)] $G$ is a connected algebraic group of dimension $p-1$.
\item[(ii)] $\Aut{\mathfrak{A}}\cong G$, the correspondence is given by sending any $\phi\in\Aut{\mathfrak{A}}$ to
$\tilde{\phi}\in G$, where $\tilde{\phi}$ is defined via $\tilde{\phi}({\mathscr{D}})=\phi\circ {\mathscr{D}}\circ \phi^{-1}$, $\forall \,\mathscr{D}\in\ggg$.
\item[(iii)] $G$ can be decomposed as $G={\mathbb{F}}^{\times}\ltimes{\mathbb{U}}$, where ${\mathbb{F}}^{\times}$ is the multiplicative group, and
$\mathbb{U}$ is the unipotent radical of $G$. More precisely, any element in $G$ is of the form $\tilde\varphi$, where $\varphi\in\Aut {\mathfrak{A}}$
is given as
$$\varphi(x)=\sum\limits_{i=1}^{p-1} a_ix^i,\,\,a_i\in {\mathbb{F}},\,\,i=1,\cdots, p-1,\,\,\text{and}\,\,\,a_1\neq 0.$$
Moreover, $\tilde{\varphi}\in {\mathbb{F}}^{\times}$ if and only if $a_i=0$ for $i=2,\cdots, p-1$. And $\tilde{\varphi}\in {\mathbb{U}}$
if and only if $a_1=1$.
\end{itemize}
\end{lem}

\begin{rem}
Lemma \ref{lem-1} is not valid for $p=3$. In fact, when $p=3$, the Witt algebra $W_1\cong\sss\lll_2$, and $\Aut(\sss\lll_2)$
has dimension $3$.
\end{rem}

\section{Description of nilpotent elements in the Witt algebra}

Keep in mind that $\ggg=W_1$ is the Witt algebra over $\mathbb{F}$. An element in $\ggg$ is called nilpotent if it is a nilpotent
operator on $\mathfrak{A}$. Set $\N=\{x\in\ggg\mid x^{[p]}=0\}$. Then $\N$ is just the set of all nilpotent elements in $\ggg$. In the
literature, $\N$ is usually called the nilpotent cone, which is a closed subvariety in $\ggg$. Then nilpotent cone $\N$ was extensively
studied by Premet in \cite{Pr}. The following result is due to Premet.

\begin{lem}(cf. \cite[Theorem 2 and Lemma 4]{Pr} or \cite[Lemma 3.1]{YS})\label{lem-2}
Keep notations as above, then the following statements hold.
\begin{itemize}
\item[(i)] The orbit $G\cdot D$ is open and dense in $\N$. Moreover, it coincides with $(\ggg\setminus\ggg_{0})\cap \N$.
\item[(ii)] We have decomposition $\N=G\cdot D\cup\ggg_1$.
\end{itemize}
\end{lem}

We are now in the position to give one of the main results describing nilpotent elements in the Witt algebra.

\begin{prop}\label{prop-1}
Let $\ggg$ be the Witt algebra with the nilpotent cone $\N$ defined as above. Then
$$D+\sum\limits_{i=0}^{p-2}k_i\,x^{i+1}D\in \N$$
if and only if the following
identity holds
\begin{equation}\label{identity-1}
k_{p-2}=\sum\limits_{i=0}^{p-3}2(i+1)k_il_{p-2-i},
\end{equation}
where $l_i$'s $(1\leq i\leq p-2)$ are defined inductively as follows.
\begin{equation}\label{identity-2}
l_1=\frac{k_0}{2},
\end{equation}
\begin{equation}\label{identity-3}
l_i=\frac{1}{i+1}\big(ik_{i-1}+\sum\limits_{j=0}^{i-2}(2j+1-i)k_jl_{i-1-j}\big),\,\,2\leq i\leq p-2.
\end{equation}
\end{prop}

\begin{proof}
(i) Suppose $X=D+\sum_{i=0}^{p-2}k_i\,x^{i+1}D\in \N$, then $X\in G\cdot D$ by Lemma \ref{lem-2}(ii),
i.e., there exists $\sigma\in G$ such that $X=\sigma(D)$. More precisely, $\sigma\in\mathbb{U}$, since $\sigma(D)-D\in\ggg_0$.
Let $Y=\sigma(xD)$. We can write down $Y$ as follows
$$Y=xD+l_1\,x^2D+l_2\,x^3D+\cdots+l_{p-2}\,x^{p-1}D.$$
It is easy to check that
\begin{eqnarray}\label{identity-4}
[X,Y]&=&[D+\sum\limits_{i=0}^{p-2}k_i\,x^{i+1}D,\,xD+\sum\limits_{j=1}^{p-2}l_j\,x^{j+1}D_j]\nonumber\\
&=& D+2l_1\,xD+\sum\limits_{t=1}^{p-3}\big((t+2)l_{t+1}+\sum\limits_{s=0}^{t-1}(t-2s)k_sl_{t-s}-tk_t\big)x^{t+1}D\\
&&+\big(2k_{p-2}+\sum\limits_{i=1}^{p-3}2(i+1)k_{p-2-i}l_i+(p-2)k_0l_{p-2}\big)x^{p-1}D.\nonumber
\end{eqnarray}

On the other hand,
\begin{eqnarray}\label{identity-5}
[X,Y]&=&[\sigma(D),\,\sigma(xD)]=\sigma([D,\,xD])=\sigma(D)=
X=D+\sum\limits_{i=0}^{p-2}k_i\,x^{i+1}D.
\end{eqnarray}

Comparing (\ref{identity-4}) with (\ref{identity-5}), we get the following relation
\begin{eqnarray}\label{identity-6}
2l_1=k_0.
\end{eqnarray}
\begin{eqnarray}\label{identity-7}
(t+2)l_{t+1}+\sum\limits_{s=0}^{t-1}(t-2s)k_sl_{t-s}-tk_t=k_t,\,\,1\leq t\leq p-3.
\end{eqnarray}
\begin{eqnarray}\label{identity-8}
2k_{p-2}+\sum\limits_{s=1}^{p-3}2(s+1)k_{p-2-s}l_s+(p-2)k_0l_{p-2}=k_{p-2}.
\end{eqnarray}
By (\ref{identity-8}), we have
\begin{eqnarray*}
k_{p-2}&=&-(p-2)k_0l_{p-2}-\sum\limits_{s=1}^{p-3}2(s+1)k_{p-2-s}l_s\\
&=&-(p-2)k_0l_{p-2}+\sum\limits_{s=1}^{p-3}2(p-1-s)k_{p-2-s}l_s\\
&=&\sum\limits_{i=0}^{p-3}2(i+1)k_il_{p-2-i}
\end{eqnarray*}
where $l_i$'s $(1\leq i\leq p-2)$ are defined by (\ref{identity-6}) and
(\ref{identity-7}). It is easy to check that (\ref{identity-6})-(\ref{identity-7}) are equivalent to (\ref{identity-2})-(\ref{identity-3}).

(ii) Let $$X=D+\sum\limits_{i=1}^{p-2}k_i\,x^{i+1}D.$$ Suppose (\ref{identity-1}) holds. We need to show that $X\in\N$. For that, set
$$Y=xD+\sum\limits_{i=1}^{p-2}l_i\,x^{i+1}D$$
with $l_i$'s defined by (\ref{identity-2}), (\ref{identity-3}). Following the same arguments as in part (i), it is a routine to check that
$[X,Y]=X$. According to \cite{De}, there exists $\sigma\in\mathbb{U}$ such that $X^{\prime}:=\sigma(X)=D+cx^{p-1}D$ for some $c\in\mathbb{F}$.
Let $Y^{\prime}:=\sigma(Y)=xD+\sum\limits_{i=1}^{p-2}l_i^{\prime}\,x^{i+1}D$. Then
$$[X^{\prime}, Y^{\prime}]=[\sigma(X), \sigma(Y)]=\sigma([X,Y])=\sigma(X)=X^{\prime}.$$
By the same arguments as in part (i), we get a similar relation as (\ref{identity-1}) on the coefficients $\{0,0,\cdots, 0, c\}$
in the expression of $X^{\prime}$ as a linear span of $\{D, xD,\cdots, x^{p-1}D\}$. This forces $c=0$. Hence $\sigma(X)=D$, i.e.,
$X=\sigma^{-1}(D)\in\N$, as desired.
\end{proof}

For any $(k_0, k_1,\cdots, k_{p-3})\in {\mathbb{F}}^{p-2}$, we define $(k_1^{\prime},\cdots, k_{p-2}^{\prime})=
(l_1,\cdots,l_{p-2})\in {\mathbb{F}}^{p-2}$ by
(\ref{identity-2})-(\ref{identity-3}). Thanks to Lemma \ref{lem-2} and Proposition \ref{prop-1}, we get the following explicit
description of nilpotent elements not contained in $\ggg_0$.

\begin{thm}\label{thm-1}
Let $\ggg$ be the Witt algebra, and $G$ the automorphism group of $\ggg$. Then
$$G\cdot D=\{aD+\sum\limits_{i=0}^{p-2}a^{-i}\,k_i\,x^{i+1}D\mid a\in {\mathbb{F}}\setminus\{0\},\, (k_0,\cdots,k_{p-3})\in
{\mathbb{F}}^{p-2}, k_{p-2}=\sum\limits_{i=0}^{p-3}2(i+1)k_ik_{p-2-i}^{\prime}\}.$$
\end{thm}

\begin{proof}
(i) Let $X\in G\cdot D$, then we can write
$$X=aD+\sum\limits_{i=0}^{p-2}b_i\,x^{i+1}D,\,\, a, b_i\in {\mathbb{F}}, i=0,\cdots, p-2,\,\,\text{and}\,\,a\neq 0.$$
Take $\sigma\in G$ such that
$$\sigma(x^iD)=a^{i-1}\,x^iD,\,\,0\leq i\leq p-1.$$
Then
\begin{equation}\label{a equality}
\sigma(X)=D+\sum\limits_{i=0}^{p-2}k_i\,x^{i+1}D,
\end{equation}
where $k_i=a^i\,b_i,\,0\leq i\leq p-2$. Since $\sigma(X)\in G\cdot D\subset\N$, we get by Proposition \ref{prop-1}
$$k_{p-2}=\sum\limits_{i=0}^{p-3}2(i+1)k_ik_{p-2-i}^{\prime}.$$
It follows from (\ref{a equality}) that
\begin{eqnarray*}
X&=&\sigma^{-1}\big(D+\sum\limits_{i=0}^{p-2}k_i\,x^{i+1}D\big)\\
&=&aD+\sum\limits_{i=0}^{p-2}a^{-i}\,k_i\,x^{i+1}D.
\end{eqnarray*}
Hence,
$$G\cdot D\subseteq\{aD+\sum\limits_{i=0}^{p-2}a^{-i}\,k_i\,x^{i+1}D\mid a\in {\mathbb{F}}\setminus\{0\},\, (k_0,\cdots,k_{p-3})\in
{\mathbb{F}}^{p-2},k_{p-2}=\sum\limits_{i=0}^{p-3}2(i+1)k_ik_{p-2-i}^{\prime}\}.$$

(ii) Let
$X=aD+\sum\limits_{i=0}^{p-2}a^{-i}\,k_i\,x^{i+1}D\in\ggg$
with
$$a\in {\mathbb{F}}\setminus\{0\},\, (k_0,\cdots,k_{p-3})\in
{\mathbb{F}}^{p-2}$$
and
$$k_{p-2}=\sum\limits_{i=0}^{p-3}2(i+1)k_ik_{p-2-i}^{\prime}.$$
Let $\sigma\in G$ such that
$$\sigma(x^iD)=a^{i-1}\,x^iD,\,\,0\leq i\leq p-1.$$
Then
$$\sigma(X)=D+\sum\limits_{i=0}^{p-2}k_i\,x^{i+1}D.$$
By Lemma \ref{lem-2} and Proposition \ref{prop-1}, $\sigma(X)\in G\cdot D$, so that
$X\in G\cdot D$. Hence,
$$\{aD+\sum\limits_{i=0}^{p-2}a^{-i}\,k_i\,x^{i+1}D\mid a\in {\mathbb{F}}\setminus\{0\},\, (k_0,\cdots,k_{p-3})\in
{\mathbb{F}}^{p-2},k_{p-2}=\sum\limits_{i=0}^{p-3}2(i+1)k_ik_{p-2-i}^{\prime}\}\subseteq G\cdot D.$$

In conclusion, combining (i) with (ii), we get
$$G\cdot D=\{aD+\sum\limits_{i=0}^{p-2}a^{-i}\,k_i\,x^{i+1}D\mid a\in {\mathbb{F}}\setminus\{0\},\, (k_0,\cdots,k_{p-3})\in
{\mathbb{F}}^{p-2},k_{p-2}=\sum\limits_{i=0}^{p-3}2(i+1)k_ik_{p-2-i}^{\prime}\},$$
as desired.
\end{proof}

As a direct consequence, we have

\begin{cor}\label{cor-1}
Keep notations as before. Then the following statements hold.
\begin{itemize}
\item[(i)] The nilpotent orbit $G\cdot D$ has dimension $p-1$.
\item[(ii)] We have the following decomposition for the nilpotent cone.
\begin{eqnarray*}
{\N}&=&\{aD+\sum\limits_{i=0}^{p-2}a^{-i}\,k_i\,x^{i+1}D\mid a\in {\mathbb{F}}\setminus\{0\},\, (k_0,\cdots,k_{p-3})\in
{\mathbb{F}}^{p-2},\,\,\text{and}\\
&&k_{p-2}=\sum\limits_{i=0}^{p-3}2(i+1)k_ik_{p-2-i}^{\prime}\}\cup \ggg_1.
\end{eqnarray*}
\end{itemize}
\end{cor}

\begin{rem}
Corollary \ref{cor-1} (i) was obtained by Premet in \cite{Pr}.
\end{rem}

Let
$$X=\sum\limits_{i=-1}^{p-2}k_ix^{i+1}D\in\ggg$$
with $k_i\in {\mathbb{F}}$, $-1\leq i\leq p-2$. It is easy to see that the matrix of $X$ relative
to the canonical basis $\{1,x,\cdots, x^{p-1}\}$ of $\mathfrak{A}$ is

\[
A=\left(\begin{array}{ccccccccccccc}
0& k_{-1} & 0 & \cdots & \cdots & \cdots & 0\\
0& k_0 & 2k_{-1} & \cdots & \cdots & \cdots & 0\\
0& k_1 & 2k_0 & \cdots & \cdots & \cdots & 0\\
\cdots & \cdots & \cdots & \cdots & \cdots & \cdots & \cdots \\
\cdots & \cdots & \cdots & \cdots & \cdots & \cdots & \cdots \\
\cdots & \cdots & \cdots & \cdots & \cdots & \cdots & \cdots \\
0& k_{p-2} & 2k_{p-3} & \cdots & \cdots & \cdots & (p-1)k_0
\end{array}\right)
\]
So, the corresponding characteristic polynomial of $A$ is

$$\aligned
(\clubsuit)\quad\quad\quad\quad |\lambda \id-A|&=\begin{vmatrix}
\lambda & -k_{-1} & 0 & \cdots & \cdots & \cdots & 0\\
0& \lambda-k_0 & -2k_{-1} & \cdots & \cdots & \cdots & 0\\
0& -k_1 & \lambda-2k_0 & \cdots & \cdots & \cdots & 0\\
\cdots & \cdots & \cdots & \cdots & \cdots & \cdots & \cdots \\
\cdots & \cdots & \cdots & \cdots & \cdots & \cdots & \cdots \\
\cdots & \cdots & \cdots & \cdots & \cdots & \cdots & \cdots \\
0& -k_{p-2} & -2k_{p-3} & \cdots & \cdots & \cdots &
\lambda-(p-1)k_0
\end{vmatrix}\cr
&= \lambda^p+ f(k_{-1},k_0,\cdots, k_{p-2})\lambda\cr
\endaligned$$
where the expression of the right hand side of $(\clubsuit)$ follows from \cite[Corollary 1]{Pr}, $\id$ represents the
identity transformation on $\mathfrak{A}$, and
$$\aligned
f(k_{-1},k_0,\cdots, k_{p-2}) &= \begin{vmatrix}
k_0 & 2k_{-1} & 0 &\cdots & \cdots & \cdots & 0\\
k_1 & 2k_0 & 3k_{-1} &\cdots & \cdots & \cdots & 0\\
k_2 & 2k_1 & 3k_0 &\cdots & \cdots & \cdots & 0\\
\cdots & \cdots & \cdots & \cdots & \cdots & \cdots & \cdots \\
\cdots & \cdots & \cdots & \cdots & \cdots & \cdots & \cdots \\
\cdots & \cdots & \cdots & \cdots & \cdots & \cdots & \cdots \\
k_{p-2} & 2k_{p-3} & 3k_{p-4} & \cdots & \cdots & \cdots & (p-1)k_0
\end{vmatrix}.\cr
\endaligned$$
Hence,
$$X=\sum\limits_{i=-1}^{p-2}k_ix^{i+1}D\in\N$$
if and only if $f(k_{-1},k_0,\cdots, k_{p-2})=0$.

According to \cite[Theorem 2]{Pr-2} and the proof of \cite[Corollary 6.10]{BFS}, there exists a homogeneous polynomial
$\psi_0\in {\mathbb{F}}[k_{-1},k_0,\cdots,k_{p-2}]$ with $\deg\psi_0=p-1$ such that $X^{[p]}=\psi_0 X$ for any
$X=\sum_{i=-1}^{p-2}k_ix^{i+1}D\in\ggg$. Hence,
$X=\sum_{i=-1}^{p-2}k_ix^{i+1}D\in\N$
if and only if $\psi_0=0$. Moreover, if $X\not\in \N$, then $g(\lambda):=\lambda^p-\psi_0\lambda$ is a minimal polynomial of $X=\sum_{i=-1}^{p-2}k_ix^{i+1}D$ as a transformation on $\mathfrak{A}$. Since $(\clubsuit)$ is a characteristic polynomial of $X$, it follows that
$g(\lambda)\mid \big(\lambda^p+f(k_{-1},k_0,\cdots, k_{p-2})\lambda\big)$. Thus, $g(\lambda)= \lambda^p+f(k_{-1},k_0,\cdots, k_{p-2})\lambda$,
i.e., $\psi_0=-f(k_{-1},k_0,\cdots, k_{p-2})$. In conclusion, $X^{[p]}=-f(k_{-1},k_0,\cdots, k_{p-2})X$ for any $X=\sum_{i=-1}^{p-2}k_ix^{i+1}D\in\ggg$.

As a direct consequence of Corollary \ref{cor-1}, we have
\begin{cor}
Keep notations as above, then
$$\aligned
 \begin{vmatrix}
k_0 & 2k_{-1} & 0 &\cdots & \cdots & \cdots & 0\\
k_1 & 2k_0 & 3k_{-1} &\cdots & \cdots & \cdots & 0\\
k_2 & 2k_1 & 3k_0 &\cdots & \cdots & \cdots & 0\\
\cdots & \cdots & \cdots & \cdots & \cdots & \cdots & \cdots \\
\cdots & \cdots & \cdots & \cdots & \cdots & \cdots & \cdots \\
\cdots & \cdots & \cdots & \cdots & \cdots & \cdots & \cdots \\
k_{p-2} & 2k_{p-3} & 3k_{p-4} & \cdots & \cdots & \cdots & (p-1)k_0
\end{vmatrix}&=0\cr
\endaligned$$
if and only if one of the following two cases occurs
\begin{itemize}
\item[(i)] $k_{-1}=k_0=0$.
\item[(ii)] $k_{-1}\neq 0$ and there exists some $(\kappa_0,\cdots, \kappa_{p-3})\in{\mathbb{F}}^{p-2}$ such that $k_i=k_{-1}^{-i}\kappa_i$ for
$0\leq i\leq p-3$ and $k_{p-2}=k_{-1}^{2-p}\sum\limits_{i=0}^{p-3}2(i+1)\kappa_i \kappa_{p-2-i}^{\prime}$.
\end{itemize}
\end{cor}

\begin{proof}
Since $X^{[p]}=-f(k_{-1},k_0,\cdots, k_{p-2})X$ for any $X=\sum_{i=-1}^{p-2}k_ix^{i+1}D\in\ggg$, it follows that
$$\aligned
f(k_{-1},k_0,\cdots, k_{p-2}) &= \begin{vmatrix}
k_0 & 2k_{-1} & 0 &\cdots & \cdots & \cdots & 0\\
k_1 & 2k_0 & 3k_{-1} &\cdots & \cdots & \cdots & 0\\
k_2 & 2k_1 & 3k_0 &\cdots & \cdots & \cdots & 0\\
\cdots & \cdots & \cdots & \cdots & \cdots & \cdots & \cdots \\
\cdots & \cdots & \cdots & \cdots & \cdots & \cdots & \cdots \\
\cdots & \cdots & \cdots & \cdots & \cdots & \cdots & \cdots \\
k_{p-2} & 2k_{p-3} & 3k_{p-4} & \cdots & \cdots & \cdots & (p-1)k_0
\end{vmatrix}&=0\cr
\endaligned$$
if and only if $$X=\sum_{i=-1}^{p-2}k_ix^{i+1}D\in\N.$$ Now the assertion follows directly from Corollary \ref{cor-1}.
\end{proof}

\section{Borel subalgebras of the Witt algebra}

As is well-known to all, Borel subalgebras play a fundamental role in the structure and representation theory of a
Lie algebra. In this section, we determine the conjugacy
classes of Borel subalgebras in the Witt algebra $\ggg$ over an algebraically closed field of prime characteristic $p>3$.

\begin{defn}
A Borel subalgebra of a Lie algebra is defined to be a maximal solvable subalgebra.
\end{defn}

Recall that $\ggg$ has a $\mathbb{Z}$-grading $\ggg=\sum\limits_{i=-1}^{p-2}\ggg_{[i]}$. Set
$${{\mathscr{B}}^+}=\ggg_0=\sum\limits_{i=0}^{p-2}\ggg_{[i]}=\sum\limits_{i=0}^{p-2}{\mathbb{F}}x^{i+1}D,$$
and
$${{\mathscr{B}}^{-}}=\sum\limits_{i=-1}^{0}\ggg_{[i]}={\mathbb{F}}D\oplus {\mathbb{F}}xD.$$
It is easy to check that ${{\mathscr{B}}^+}$ and ${{\mathscr{B}}^-}$ are Borel subalgebras. Moreover, they are not conjugate to each other.
We call ${{\mathscr{B}}^+}$ and ${{\mathscr{B}}^-}$  standard Borel subalgebras. The following result asserts that there are only two
conjugacy classes of Borel subalgebras under the automorphism group $G$ of $\ggg$. They are represented by ${{\mathscr{B}}^+}$ and ${{\mathscr{B}}^-}$.

\begin{thm}\label{thm-2}
Let $\ggg=W_1$ be the Witt algebra over an algebraically closed field $\mathbb{F}$ of characteristic $p>3$. Then
any Borel subalgebra in $\ggg$ is conjugate to a standard Borel subalgebra under the action of the automorphism group of $\ggg$.
\end{thm}

\begin{proof}
Let $B$ be a Borel subalgebra of $\ggg$. We divide the discussion into two cases.

\textbf{Case 1:} $B\subseteq \ggg_0$.

In this case, $B={\mathscr{B}}^{+}$, since $\ggg_0={\mathscr{B}}^{+}$.

\textbf{Case 2:} $B\nsubseteqq \ggg_0$.

In this case, we claim that $B$ is conjugate to ${\mathscr{B}}^{-}$. For that, we consider the intersection of $B$ with the
nilpotent cone $\N$.

(i) $B\cap {\N}\neq 0$.

In this situation, let $0\neq u\in B\cap {\N}$, then $u\in G\cdot D$ or $u\in\ggg_1$ by Lemma \ref{lem-2}.

(i-1) Suppose $u\in G\cdot D$, then there exists $\sigma\in G$ such that $D\in \sigma(B)$. This implies that $\sigma(B)={\mathscr{B}}^{-}$.
Indeed, if there exists $v\in \sigma(B)\setminus {\mathscr{B}}^{-}$, we can write
$$v=\sum\limits_{i=-1}^{p-2}a_ix^{i+1}D$$
with $a_i\neq 0$ for some $i>0$. Set
$$j=\max\,\{l>0\mid a_l\neq 0\}.$$
Then
$$(\ad\,D)^{j-1}(v)=\frac{(j+1)!}{2}a_j\,x^2D+j!\,a_{j-1}\,xD+(j-1)!\,a_{j-2}\,D\in \sigma(B),$$
$$(\ad\,D)^j(v)=(j+1)!\,a_j\,xD+j!\,a_{j-1}\,D\in \sigma(B).$$
Therefore, $\sigma(B)$ contains a semisimple subalgebra span$_{\mathbb{F}}\{D,xD,x^{2}D\}\cong\sss\lll_2$. It contradicts
with the solvability of the subalgebra $\sigma(B)$.

(i-2) Suppose $u\in \ggg_1$, then we will get a contradiction. Since $B\nsubseteqq\ggg_0$, there exists $u^{\prime}\in B\setminus \ggg_0$. 
According to \cite{De},
we can find $\sigma\in G$ such that $\sigma(u^{\prime})=D+c\,x^{p-1}D$ for some $c\in\mathbb{F}$. Write
$$\sigma(u)=\sum\limits_{j=i}^{p-2}k_jx^{j+1}D\,\,\text{for}\,\,\text{some}\,\,i\geq 1,\,\,\text{and}\,\,k_i\neq 0.$$
Then
$$\big(\ad(\sigma(u^{\prime}))\big)^{i-1}\big(\sigma(u)\big)=\frac{(i+1)!}{2}\,k_i\,x^2D+w_2\in\sigma(B),$$
and
$$\big(\ad(\sigma(u^{\prime}))\big)^{i}\big(\sigma(u)\big)=(i+1)!\,k_i\,xD+w_1\in\sigma(B),$$
where $w_1\in\ggg_1,\,w_2\in\ggg_2$. Hence, $\sigma(B)$ contains a subalgebra
$$\text{span}_{\mathbb{F}}\{D+c\,x^{p-1}D,\,\,
(i+1)!\,k_i\,xD+w_1,\,\, \frac{(i+1)!}{2}\,k_i\,x^2D+w_2\}$$
which is not solvable by a direct check. This contradicts
with the solvability of the subalgebra $\sigma(B)$.

(ii) $B\cap \N=0$.

In this situation, $B\cap \ggg_1=0$, so that $\dim\,B\leq \dim\,\ggg-\dim\,\ggg_1=2$.

(ii-1) If $\dim\,B=1$. We can write $B=\text{span}_{\mathbb{F}}\{X\}$, where $X\in\ggg\setminus\ggg_0$. According to \cite[Theorem 2]{Pr-2},
$X^{[p]}=cX$ for some $c\in\mathbb{F}$. Since $B\cap {\N}=0$, we get $c\neq 0$, i.e., $X$ is a semisimple element. It follows that
$\ggg$ can be decomposed as a direct sum of eigensubspaces of $\ad\,X$. Take any eigenvector $Y\neq X$, then $\text{span}_{\mathbb{F}}\{X,\,Y\}$
is a two-dimensional solvable subalgebra containing $B$. It contradicts with the maximality of $B$ as a solvable subalgebra.

(ii-2) If $\dim\,B =2$. By the structure of two-dimensional Lie algebras and the assumption above, we can choose a basis $\{X, Y\}$ of $B$ with
$$X=D+\sum\limits_{i=0}^{p-2}k_i\,x^{i+1}D,\,\,Y=xD+\sum\limits_{i=1}^{p-2}l_i\,x^{i+1}D,\,\,\text{and}\,\,[X,Y]=X.$$
A similar argument as part (i) in the proof of Proposition \ref{prop-1} yields $X\in {\N}$. It contradicts with the assumption
$B\cap \N=0$.

In conclusion, $B$ is conjugate to $\mathscr{B}^+$ or $\mathscr{B}^{-}$. More precisely, If $B\subseteq\ggg_0$, then $B$ is conjugate to $\mathscr{B}^+$.
While if $B\nsubseteqq \ggg_0$, then $B$ is conjugate to $\mathscr{B}^-$. This completes the proof.
\end{proof}

\begin{rem}
The restriction on the characteristic $p$ of the base field $\mathbb{F}$ is necessary. Indeed, Theorem \ref{thm-2} is not valid for $p=2, 3$.
When $p=2$, the Witt algebra $W_1$ is a solvable Lie algebra. Hence, there is only one Borel subalgebra which is just $W_1$ itself. When
$p=3$, the Witt algebra $W_1\cong\sss\lll_2$. Classical result implies that all Borel subalgebras are conjugate.
\end{rem}

The non-classical Lie algebra of rank 1, i.e., the Witt algebra $W_1$, can be generalized to higher rank $n$, i.e., the
Jacobson-Witt algebra $W_n$. Let ${\mathfrak{A}}_n={\mathbb{F}}[x_1,\cdots, x_n]/(x_1^p,\cdots, x_n^{p})$ be the truncated polynomial algebra
 of $n$ indeterminates. By definition, the Jacobson-Witt algebra $W_n$ is the derivation algebra of ${\mathfrak{A}}_n$. Let $D_i\,(1\leq i\leq n)$ be the partial derivation with respect to the $i$-th indeterminate defined via $D_i(x_j)=\delta_{ij}$ for $1\leq i,j\leq n$. Then by \cite[\S\,4.2]{SF},
$W_n=\text{span}_{\mathbb{F}}\{\sum\limits_{i=1}^n f_iD_i\mid f_i\in {\mathfrak{A}}_n\}.$
According to \cite{De}, under the automorphism group $G=\Aut W_n$ of $W_n$, any Cartan subalgebra in $W_n$ is conjugate to one of the following subalgebras.
\begin{eqnarray*}
T_0&=&\text{span}_{{\mathbb{F}}}\{(1+x_1)D_1,\,(1+x_2)D_2,\,\cdots,(1+x_n)D_n\}\\
T_1&=&\text{span}_{{\mathbb{F}}}\{x_1D_1,\,(1+x_2)D_2,\,\cdots,(1+x_n)D_n\}\\
\cdots&\cdots&\cdots\cdots\cdots\cdots\cdots\cdots\cdots\cdots\cdots\cdots\cdots\cdots\\
T_{n-1}&=&\text{span}_{{\mathbb{F}}}\{x_1D_1,\, x_2D_2,\,\cdots,x_{n-1}D_{n-1},\,(1+x_n)D_n\}\\
T_n&=&\text{span}_{{\mathbb{F}}}\{x_1D_1,\,x_2D_2,\,\cdots,x_nD_n\}\\
\end{eqnarray*}

One can expect the intimate relation between Cartan subalgebras and Borel subalgebras. We conclude this section with the following conjecture.

\begin{con}\label{conj}
Let $\ggg=W_n$ be the Jacobson-Witt algebra of rank $n$. Then there are $n+1$ conjugacy classes of Borel subalgebras in $\ggg$ under the
action of the automorphism group of $\ggg$.
\end{con}

\begin{rem}
Conjecture \ref{conj} is true for the Witt algebra $W_1$ by Theorem \ref{thm-2}.
\end{rem}

\vspace{0.5cm}

\subsection*{Acknowledgements}
This work was done during the visit of the authors to the department of mathematics in the University
of Kiel in 2012-2013. The authors would like to express their sincere gratitude to professor Rolf Farnsteiner for his invitation, hospitality and stimulating discussions. The authors also thank the department of mathematics for providing an excellent atmosphere. Many thanks are also given to professor Bin Shu who stimulated us to consider Borel subalgebras of a general Lie algebra during
the summer of 2012 in Shanghai.

\end{document}